\documentclass[12pt,english]{article}
\usepackage[T1]{fontenc}
\usepackage[latin9]{inputenc}
\usepackage{geometry}
\usepackage{amsthm}
\usepackage{amsmath}
\usepackage{amssymb}
\usepackage{amsfonts}
\usepackage{color}

\makeatletter 

\theoremstyle{plain}
\newtheorem{thm}{\protect\theoremname}
  \theoremstyle{definition}
  \newtheorem{defn}[thm]{\protect\definitionname}
  \theoremstyle{remark}
  \newtheorem{rem}[thm]{\protect\remarkname}
  \theoremstyle{plain}
  \newtheorem{prop}[thm]{\protect\propositionname}
     \theoremstyle{plain}
   
  \theoremstyle{definition}
     \newtheorem{notation}[thm]{Notation}
  \newtheorem{example}[thm]{\protect\examplename}
     \theoremstyle{plain}
  \newtheorem{corollary}[thm]{\protect\corollaryname}
     \theoremstyle{plain}
  \theoremstyle{plain}
  \newtheorem{conjecture}[thm]{Conjecture}
  \newtheorem{alg}[thm]{Algorithm}

\DeclareMathOperator{\codim}{codim}

\DeclareMathOperator{\LL}{LL}
\DeclareMathOperator{\mldegree}{MLdegree}
\DeclareMathOperator{\mltable}{MLtable}
\DeclareMathOperator{\Gr}{Gr}

\newcommand{\cH}{\mathcal{H}}

\newcommand{\cL}{\mathcal{L}}

\newcommand{\cM}{\mathcal{M}}
\newcommand{\C}{\mathbb{C}}
\newcommand{\R}{\mathbb{R}}
\newcommand{\PP}{\mathbb{P}}
\newcommand{\shortM}{M}

\makeatother

\usepackage{babel}
  \providecommand{\definitionname}{Definition}
  \providecommand{\propositionname}{Proposition}
  \providecommand{\remarkname}{Remark}
\providecommand{\theoremname}{Theorem}
\providecommand{\examplename}{Example}
\providecommand{\corollaryname}{Corollary}
\begin{document}

\newcommand{\elizabeth}[1]{\textcolor{forestgreen}{\textbf{E: #1}}}


\title{Maximum likelihood geometry \\in the presence of data zeros}
\author{Elizabeth Gross\thanks{Department of Mathematics, North Carolina State University, Raleigh, NC; {\tt eagross@ncsu.edu}. This work was partially supported by NSF award DMS--1304167.} \ and Jose Israel Rodriguez\thanks{
Department of Mathematics,
University of California at Berkeley,
Berkeley, CA 94720; 
  {\tt jo.ro@berkeley.edu.}
 The second author is supported  by the US National Science Foundation DMS-0943745.}}

\date{30 April 2014}


\maketitle
\begin{abstract}

Given a statistical model, the maximum likelihood degree is the number of complex solutions to the likelihood equations for generic data.  We consider discrete algebraic statistical models and study the solutions to the likelihood equations when the data contain zeros and are no longer generic.  Focusing on sampling and model zeros, we show that, in these cases, the solutions to the likelihood equations are contained in a previously studied variety, the likelihood correspondence. 
The number of these solutions give a lower bound on the ML degree, and the problem of finding critical points to the likelihood function can be partitioned into smaller and computationally easier problems involving sampling and model zeros.  We use this technique to compute a lower bound on the ML degree for $2 \times 2 \times 2 \times 2$ tensors of border rank $\leq 2$ and $3 \times n$ tables of rank $\leq 2$ for $n=11, 12, 13, 14$, the first four values of $n$ for which the ML degree was previously unknown.

\end{abstract}

\section{Introduction}
The method of maximum likelihood estimation for a statistical model $\cM$ and an observed data vector $u \in \R^{n+1}$ involves maximizing the likelihood function $l_u$ over all distributions in $\cM$. 
This involves understanding the zero-set of a system of equations, and, thus, when the models of interest are algebraic, the process lends itself  to investigation using algebraic geometry. 
In fact, likelihood geometry has been studied in a series of papers in the field of algebraic statistics beginning with \cite{HKS05,CHKS06}. Subsequent papers include \cite{BHR07,HS10,GDP12,HRS12,Uhl12}  covering both discrete and continuous models. In addition, the complexity of finding critical points using ideal theoretic methods has been explored in \cite{FSS12}.  In this paper, we look at discrete models and the case where the observed data vector contains zero entries.

 In \cite{HKS05}, Ho{\c{s}}ten, Khetan, and Sturmfels introduce the likelihood locus and its associated incidence variety for discrete statistical models. 
 In \cite{HS13}, Huh and Sturmfels study this incidence variety further under the name of the \emph{likelihood correspondence}.  
 Given a discrete algebraic statistical model with sample space of size $n+1$ and Zariski closure $X$, the likelihood correspondence $\cL_X$ is a closed algebraic subset of $\PP^{n} \times \PP^{n}$.
 We follow \cite{HS13} and write  $\PP^n \times \PP^n$ as $\PP_p^n \times \PP_u^n$ to emphasize that the first factor is the probability space, with homogeneous coordinates $p_0, p_1, \ldots, p_n$, and the second factor is the data space, with homogeneous coordinates $u_0, u_1, \ldots, u_n$.  
 In this paper, we are concerned with special fibers of the projections  $pr_1: \cL_X \to \PP^{n}_p$ and $pr_2: \cL_X \to \PP^{n}_u$.  Specifically, we set out to understand $pr_2^{-1}(u)$ when $u$ contains zero entries and show how our understanding of $pr_2^{-1}(u)$ yields information about generic fibers of $pr_2$.  The degree of a generic fiber of $pr_2$ is known as the \emph{ML degree} (maximum likelihood degree) of $X$.

A statistical model $\mathcal{M}$ is a subset of the probability simplex
 $\Delta_n = \{ (p_0, p_1, \ldots, p_n) \in \R_{\geq 0}^{n+1} \ | \sum_{i=0}^n p_i = 1\}.$ 

Given positive integer data $u\in\mathbb{Z}_{\geq 0}^{n+1}$,
the \textit{maximum likelihood estimation} problem is to determine
 $\hat{p}\in {\cal M}$ that maximizes
the likelihood function 
\[
l_{u}=p_{0}^{u_{0}}p_{1}^{u_{1}}\cdots p_{n}^{u_{n}}
\]
restricted to ${\cal M}$. The point $\hat p\in{\cal M}$ is called the 
\textit{maximum likelihood estimate}, or MLE.
The family of models we are interested in are \emph{algebraic statistical} models, which are
defined by the vanishing of polynomial equations restricted to the
probability simplex.  

To use algebraic methods, we consider
points of ${\cal M}\subset\mathbb{R}^{n+1}$ as representatives
of points in $\mathbb{P}^{n}$ and study the Zariski closure $\overline{\cM}=X\subset\mathbb{P}^{n}$.
This makes the problem easier by relaxing the nonnegative and real constraints, which allows us to obtain an understanding about the number of possible modes of the likelihood surface.
 There are subtleties when performing this relaxation as mentioned for example in \cite{KRS} related to the the boundary of the model. 
 
 Let $p_{+}:=p_{0}+p_{1}+\cdots+p_{n}$ and $\cH_n$ be the set of points where $p_+ p_0 p_1 \cdots p_n$ equals zero.
With algebraic methods, our goal is to determine all complex critical
points of $L_{u}:=l_{u}/p_{+}^{u_{+}}$ when restricted to $X_{reg}\backslash\cH_n\subset\mathbb{P}^{n}$,
where  $X_{reg}$ is the set of regular points of $X$.  We work with $L_u$ since it is a 
function on $\PP^n$ (see Section 2.2 in \cite{DSS09}).

A point $p\in X_{reg}$
is said to be a \emph{critical point} if the gradient of $L_{u}\left(p\right)$ is
orthogonal to the tangent space of $X$ at $p$, that
is 
$\nabla L_{u}\left(p\right)\perp T_{p}X.$

If the maximum likelihood estimate  $\hat{p}$ for the data vector $u$ is in the interior of ${\cal M}$,
then $\hat{p}$ will be a critical point of $L_{u}$ over $X$. By determining the critical points of $L_u$ on $X$, we  find all local maxima of $l_u$  on $\cal M$.

When the data vector $u$ contains zero entries, each zero entry is called  either a sampling zero or a structural zero in the statistics literature. Considering $u$ as a flattened contingency table, a sampling zero at $u_i$ occurs when no observations fall into cell $i$ even though $p_i$ is nonzero. 
A structural zero occurs at $u_i$ when the probability of an observation falling into cell $i$ is zero.  
Structural and sampling zeros occur commonly in practice, for example, in large sparse data sets (for more on sampling and structural zeros see \cite{BFH75}[\S 5.1.1]). 

The terms ``sampling zero'' and ``structural zero'' are denotationally about contingency tables, but they also carry implications about $X$ as well. For example, the  term ``structural zero'' connotes that maximum likelihood estimation should proceed over a projection of $X$ (see \cite{Rap06}). Due to this secondary definition imparted to the term ``structural zero," and in view of the fact that this study is concerned with the intersection of $X$ with the hyperplane $p_i=0$ as opposed to the projection of $X$, we introduce the definition of a model zero.



\begin{defn}\textbf{[Model zeros]}  Given a model $\cM$ with $\overline{\cM}=X \subset \PP^n$ and data vector $u$ with $u_i=0$, a \emph{model zero} at cell $i$ is a zero such that the maximum likelihood estimate $\hat p$ for $u$ is a critical point of $L_u$ over $X \cap \{p_i = 0\}$.
\end{defn}

\begin{rem} For the remainder of the paper, we will use ``structural zero" to mean a zero at cell $i$ such that maximum likelihood estimation proceeds over the projection of $X$ onto all coordinates except the $i$th coordinate and  $p_i=0$.
\end{rem}

In this paper we explore the algebraic considerations of maximum likelihood estimation when the data contains sampling and model zeros.  In  Theorem \ref{thm:subproblems} of Section 3,  we show how solutions to the maximum likelihood estimation problems for data with zeros on $X$ are contained in the likelihood correspondence of $X$.  
This result gives statistical meaning to the likelihood correspondence when $u_i$ is equal to zero and we can use Theorem \ref{thm:subproblems} to compute a lower bound on the ML degree of a variety $X$.

 
This paper is organized as follows. In Section 2, we give preliminary definitions and introduce a square parameterized system called the Lagrange likelihood equations.  Proposition \ref{prop:properties} describes the properties of the Lagrange likelihood equations that will be referenced in later sections.  

In Section 3, we discuss how sampling and model zeros change the maximum likelihood problem. 
Theorem \ref{thm:subproblems} describes the special fiber $pr_2^{-1}(u)$ when $u$ contains zero entries.  
We use this theorem to give a lower bound on the ML degree of $X$.  The section continues with exploring how solutions to the Lagrange likelihood equations partition into solutions for different maximum likelihood estimation problems for sampling and model zeros; these partitions are captured in the ML tables introduced in this section.  We end this section by fully characterizing the ML degree for different sampling and model zero configurations of a generic hypersurface of degree $d$ in $\PP^n$.

We conclude with Section 4, which includes examples, timings, and applications.  In Section 4.1, we illustrate the techniques from the previous sections and report on computational timings to show the advantages of working with data zeros.  In this section, we give a lower bound on the ML degree for $3 \times n$ tables of rank $\leq 2$ for $n=11, 12, 13, 14$, these bounds give further evidence for Conjecture 4.1 in \cite{HRS12}. 
In Section 4.2, Procedure \ref{alg:1} gives a method to find critical points of $L_u$ over $X$ by computing the critical points of $L_u$ when $u$ contains model zeros; in most cases, such solutions should be easier to compute since there are less variables to consider. In Section 4.3, we extend maximum likelihood duality to $u$ with zero entries.  While in Sections 4.4. and 4.5, we look at tensor and Grassmannian examples respectively.

\section{Equations and ML degree}
The \emph{maximum likelihood degree} (ML degree) of a variety $X \subset \PP^n$ is defined as the number of critical points of the likelihood function $L_u$ on $X_{reg} \setminus \cH_n $ for generic data $u$ \cite{CHKS06}.  The ML degree of $X$ quantifies the algebraic complexity of the maximum likelihood estimation problem over the model $\cM$, indicating how feasible symbolic algebraic methods are for finding the MLE.  The ML degree has an explicit interpretation in numerical algebraic geometry as well.   Assuming that the \emph{ab initio} stage of a coefficient-parameter homotopy has been run \cite{SW05}[\S 7] the ML degree is the number of paths that need to be followed for every subsequent run.

For each $u$, all critical points of $L_u$ over $X$ form a variety.  Thus, by varying $u$ over $\PP_u^n$ we obtain a family of projective varieties with base $\PP_u^n$.  In algebraic geometry, the natural way to view this family of parameterized varieties is as a subvariety $\cL_X$ of the product variety $\PP_p^n \times \PP_u ^n$ where the elements of the family are the fibers of the canonical projection $pr_2: \PP_p^n \times \PP_u ^n \to \PP_u^n$ over the points $u$ in $\PP^n_u$.  The subvariety $\cL_X$ is called the likelihood correspondence \cite{HS13}, which is the closure in $\PP_p^n \times \PP_u^n$ of 
$$\{ (p,u) : p \in X_{reg} \setminus \cH_n \text{ and dlog}(L_u) \text{ vanishes at } p\}.$$
When $X$ is irreducible, the likelihood correspondence is an irreducible variety of dimension $n$.

Just as we can talk about a parameterized family of varieties, we can also talk about a parameterized system of polynomial equations. For us, a \textit{parameterized polynomial system} is a family $\mathcal{F}$ of polynomial equations in the variables $p_0, \ldots, p_n$ and the parameters $u_0, \ldots, u_n$. 
A member of the family is chosen by assigning a complex number to each parameter $u_i$. If $u$ is a generic vector in $\PP^{n}$,  we  call the resulting system \emph{generic}.  A system of equations is said to be $\emph square$ if the number of unknowns (variables) equals the number of equations of the system. 
Algebraic homotopies are an effective 
way to solve many members of a family $\mathcal{F}$.  
By solving a generic member of the family, we  determine the solutions to another  system of the family using a $\emph {coefficient-parameter homotopy}$ (see \cite{MS89}), thus, this viewpoint can be computationally advantageous in applications where one has to solve the same system for many different parameter values.

In this section, we define a parameterized square system of polynomial equations called the Lagrange likelihood equations. The Lagrange likelihood equations for a variety $X\subset\mathbb{P}^n $ of codimension $c$  consists of $n+1+c$ equations. 
The  $n+1+c$ unknowns are $p_0,p_1,\dots p_n,\lambda_1,\dots,\lambda_c$ and the parameters are $u_0,\dots,u_n$.  The advantage of the Lagrange likelihood equations, in addition to being a parameterized square system, is that properties of a point $(p, u)$ in the likelihood correspondence become apparent.  These properties are summarized in Proposition \ref{prop:properties}.

\begin{defn}\textbf{[Lagrange likelihood equations]}
Suppose $h_1,\dots,h_c$ is a reduced regular sequence of homogeneous polynomials, and $X$ is an irreducible component of the projective variety defined by $h_1,\dots,h_c$ with codimension $c$.
The Lagrange likelihood equations of $X$  denoted by $\LL(X,u)$
are
\begin{equation}\label{eq:LLE1}
\begin{array}{c}
h_{1}=h_2=\cdots=h_{c}=0
\end{array}
\end{equation}
\begin{equation}\label{eq:LLE2}
\begin{array}{c}
\left(u_{+}p_{i}-u_{i}\right)=p_i \left( \lambda_{1}\partial_{i}h_{1}+\lambda_{2}\partial_{i}h_{2}+\cdots +\lambda_c\partial_{i}h_{c}\right)
\end{array}\text{for } i=0, \ldots, n
\end{equation}
\end{defn}

If $X$ is a complete intersection, then $h_1, \ldots, h_c$ are  minimal generators of $I(X)$.  Otherwise, in order to satisfy the conditions imposed on $X$, one can choose $h_1, \ldots, h_c$ to be 
$c$ random linear combinations of  minimal generators of $I(X)$.

\begin{prop} \label{prop:properties}
The Lagrange likelihood equations have the following properties. 
\begin{enumerate}
\item If $(p,\lambda)$ is a solution of $\LL\left(X,u\right)$ and $u_+\neq 0$, then $\sum p_{i}=1$. 
\item If $p_{i}=0$, then $u_{i}=0$.
\item If the point $p$ is a critical point of $L_{u}$ restricted to $X_{reg}\backslash \cH_n$, then there
exists an unique $\lambda$ such that $\left(p,\lambda\right)$ is a
solution to $\LL\left(X,u\right)$. 
\item If $p\in X_{reg}\backslash \cH_n$ and $\left(p,\lambda\right)$ is a regular isolated solution to $\LL\left(X,u\right)$, then $p$ is a critical point of
$L_{u}$ on $X_{reg}\backslash \cH_n$. 
\item For generic choices of $u$, the number of solutions of $\LL\left(X,u\right)$
with $p\in X_{reg}\backslash \cH_n$ equals the ML degree of $X$. 
\end{enumerate}

\end{prop}
 
\begin{proof} 
To arrive at property (1), we sum the equations of \eqref{eq:LLE2} to get
\begin{multline}
\sum_{i=0}^n( u_+ p_i-u_i- p_i(\lambda_1\partial_{i}h_{1}+\cdots+\lambda_n \partial_{i}h_{c}))
= \sum_{i=0}^n p_i u_+-u_+=u_+(\sum_{i=0}^n p_i-1).
\end{multline}
The first equality above follows by Euler's relation of homogeneous polynomials. 

The implication stated in property (2) is clearly seen by setting $p_i$ equal to zero in the $i$th equation of Equations \eqref{eq:LLE2}.
 
For properties (3) and (4), we note that, as discussed in \cite{DR12}, $p \in X$ is a critical point of $L_u$ on $X$ if and only if the linear subspace $T_p^{\perp}$ contains the point
\[ \left( \frac{u_0}{p_0}- \frac{u_{+}}{p_{+}}: \ldots : \frac{u_n}{p_n} - \frac{u_{+}}{p_{+}} \right).\]  
When $X$ is of codimension $c$, this implies that $p \in X_{reg} \setminus \cH_n$ is a critical point for $L_u$ on $X$  if and only if there exist unique $\lambda_1, \ldots ,\lambda_c \in \C$ such that for all $0 \leq i \leq n$,
\[\frac{u_i}{p_i}- \frac{u_{+}}{p_{+}}= \lambda_1 \cdot \partial_i h_1 + \ldots + \lambda_c \cdot \partial_i h_c. \]
The Langrange likelihood equations are a restatement of this condition with the  denominators cleared.

For Property (5), we homogenize the Lagrange likelihood equations using $p_+$ and $u_+$ so that each equation is homogeneous in both the coordinates $p_0, \ldots, p_n$ and  the coordinates $u_0, \ldots, u_n$,  $\lambda_1, \ldots, \lambda_c$, then the Lagrange likelihood equations define a variety $Y$ in the product space $\PP_p^n \times \PP_{u,\lambda}^{n+c}$. Intersecting $Y$ with $X\times\PP^{n+c}$ gives us a new variety $\hat \cL_X$.
 Let $\pi$ be the projection
\begin{align*}
\pi: \mathbb{P}^{n}\times\mathbb{P}^{n+c} &\to  \mathbb{P}^{n}\times\mathbb{P}^{n}\\
(p,(u:\lambda)) &\mapsto (p, u).
\end{align*}
By properties (3) and (4), the map $\pi$ restricted to $\hat \cL_{X}$ is a birational map between $\hat \cL_{X}$ and $\cL_{X}$.

Since the map $pr_2: \cL(X) \to \mathbb{P}^n_{u}$ is generically finite to one (by Theorem 1.6 in \cite{HS13}) with degree equal to the MLdegree($X$), then $pr_2 \circ \pi: \hat \cL(X) \to \mathbb{P}^n_u$ is generically finite to one with degree equal to MLdegree($X$). This gives us the statement of Property (5) as desired.
\end{proof}

By the proof above for Proposition \ref{prop:properties}, we see that
 $$\cL(X)= \pi(\hat \cL_X).$$
The implication of this equality is that by studying the Lagrange likelihood equations, we are in fact studying fibers of the projection $pr_2: \PP_p^{n} \times \PP_u^n \to   \PP_u^{n}$.

\begin{rem}\label{parasites}
When $X$ is not a complete intersection, the Lagrange likelihood equations may have extraneous solutions that are \textit{not} points in $X$. These extraneous solutions can be handled by either filtering the solutions using a membership test or by replacing equations in \eqref{eq:LLE1} with a full set of defining equations for $X$. 
\end{rem}


We conclude this section with an example of using the Lagrange likelihood equations to find critical points of $L_u$.
\begin{example}
Let $X=\Gr_{2,6}\subset\PP^{14}$ be the variety defined by 
\[p_{ij}p_{kl}-p_{ik}p_{jl}+p_{ij}p_{jk},\quad1\leq i<j<k<l\leq 6.\]
The Grassmannian  $\Gr_{2,6}$ parameterizes lines in the projective
space $\mathbb{P}^{5}$. It has codimension $6$ and is not a complete intersection. 
However, the $6$ polynomials $h_1,\dots,h_6$ below 
\[
\begin{array}{cc}
p_{36}p_{45}-p_{35}p_{46}+p_{34}p_{56},& p_{25}p_{34}-p_{24}p_{35}+p_{23}p_{45},\\
p_{15}p_{34}-p_{14}p_{35}+p_{13}p_{45}, & p_{26}p_{45}-p_{25}p_{46}+p_{24}p_{56},\\
p_{16}p_{45}-p_{15}p_{46}+p_{14}p_{56}, & p_{14}p_{23}-p_{13}p_{24}+p_{12}p_{34}\,
\end{array}
\]
define a reducible variety that has $\Gr_{2,6}$ as an
irreducible component (the other components live in the 
coordinate hyperplanes). 
The system of equations $\LL\left({X},u\right)$ consists of $21$ equations: the $6$ equations $h_i=0$ for $i=1,\dots,6$ and the $15$ equations below given by $1\leq i<j$, 
$${u_{ij} }- {u_{+}p_{ij}}  =  p_{ij}\left(\lambda_1 \cdot \frac{\partial h_1}{\partial p_{ij}} + \ldots + \lambda_6 \cdot \frac{\partial h_6}{\partial p_{ij}}\right).$$

Solving $\LL\left({X},u\right)$,  using {\tt Bertini}, we find $156$ regular isolated solutions with $p\in X$. 
 Thus, by Proposition \ref{prop:properties} the ML degree of $X$ is 156.

\end{example}

\section{Sampling and model zeros}

In this section, we determine what happens when the data vector $u$ contains zero entries.  By understanding the  maximum likelihood estimation problems for sampling and model zeros we gain insight into the ML degree of a variety $X$.

For a subset $S\subseteq\{0, 1, \ldots, n\}$,
we define 
$$U_{S}:=\{u\in\mathbb{P}^{n} \mid u_{i}=0\,\text{if }i\in S\text{ and nonzero otherwise}\}.$$
For ease of notation, we  define $U:=U_{\emptyset}$.


The set $U_{S}$ specifies which entries of the data vector are zero, each zero entry can be either a sampling zero or model zero. A sampling zero at cell $i$ changes the likelihood function since the monomial $p_i^{u_i}$ no longer appears in $l_u$.  In the case of a model zero at cell $i$, the model zero is not considered as part of the data, and thus, the likelihood function is changed as well: $p_i^{u_i}$ no longer appears in the function and $p_i$ is set to zero in $p_+$.  Below, we make precise how the maximum likelihood estimation problem changes in the presence of model zeros and sampling zeros  and describe the maximum likelihood estimation problem on $X$ for data $u\in U_{S}$
with model zeros $R$.

Let $S \subseteq \{0,1, \ldots, n \}$ and $R \subseteq S$ and consider the following modified likelihood function
$$L_{u,S}:=\prod_{i\not \in S}p_i^{u_i}/p_+^{u_+}.$$
The set  $X_{R}:=X\cap\{p \in \PP^n \mid p_{i}=0\text{ for all }i\in R\}$ will be called the \emph{model zero variety} for $X$ and $R$.
 We consider $X_R$ as a projective variety in $\PP^{n-|R|}$  and define  $\cH_R$ as the set of points in $\PP^{n-|R|}$ where $\left( \prod_{i\not\in R} p_i \right) \cdot p_+$ vanishes. The model zero variety $X_R$ is called \emph{proper} if  the codimension of  $X_R \subset\PP^{n-\mid R\mid}$ equals the  codimension of $X\subset\PP ^n$.  
 \begin{defn}
The \emph{maximum likelihood estimation problem on $X$ for data $u\in U_{S}$
with model zeros $R$}, denoted $ML_{R,S}$, is to determine the critical points of $L_{u,S}$ on $X_R\setminus \cH_R$.  The MLdegree $\left(X_{R},S\right)$ is defined to be the number
of critical points of $L_{u,S}$ on $X_{R}\setminus \cH_R$ for generic $u \in U_S$ when $X_R$ is proper and zero otherwise. 
\end{defn}
In terms of the likelihood correspondence, the MLdegree$\left(X_{R},S\right)$ 
is the cardinality  of the subset of points $(p,u)$ of $pr_2^{-1}(u)$  such that $p_i=0$ for all $i\in R$ for generic  $u\in U_S$.
Whenever $R=S$, then $\mldegree(X_R,S)$ simply equals $\mldegree(X_R\subset\PP^{n-\mid R\mid})$.
In terms of optimization, the $\mldegree\left(X_{R},S\right)$ gives an upper bound on the local maxima  of
$l_{u,S}:=\prod_{i \notin S} p_i^{u_i}$ on  ${\cal M }\cap \{p_i=0 \text { for all }i\in R \}$. 

\medskip

Next,  we take the time to explain the subtleties of sampling zeros, model zeros, and structural zeros.
When given a model $\mathcal{M}$ with closure $X$ and structural zeros $R$, common practice is to optimize $l_{u,R}$ restricted to $\pi_R(X)$, the closure of the \emph{projection} of $X$ onto all coordinates not indexed by $R$ \cite{BFH75}\cite{Rap06}. In contrast, given a model $\mathcal{M}$ with closure $X$ and \emph{model} zeros $R$, the goal is to optimize $l_{u,R}$  restricted to $X_R$.  In general, $\pi_R(X) \neq X_R$, and so, the number of critical points will differ.  We illustrate the differences in the next example. 


\begin{notation} We use $S$ to denote the indices of the data zeros in $u$ and $R \subset S$ to denote the indices of the model zeros.  While we defined $S \subset \{0, 1, \ldots, n \}$, in some examples, it is more natural to index the entries of $u$ by ordered pairs.  In this case, $S$ will be a set of ordered pairs indicating the positions of the data zeros and $R$ will be a set of ordered pairs indicating the positions of the model zeros.
\end{notation}






\begin{example}\textbf{[Model, sampling, and structural zeros]}\label{ex:allGood}
Let $X$ denote the set of $3\times 4$ matrices of rank $2$ in $\PP^{11}$.  The defining ideal of $X$ is generated by the four $3 \times 3$ minors of $p$.
The ML degree of $X$ is $26$. 
 
Now let $u_{11}$ be a model zero in the contingency table $u$.  In this case, $R=\{(1,1)\}$ and the defining ideal of $X_R$ is
\begin{align*}
I(X_R)= \langle & p_{12}p_{21}p_{33}+p_{12}p_{23}p_{31}+p_{13}p_{21}p_{32}-p_{13}p_{22}p_{31}, \\
&p_{12}p_{21}p_{34}+p_{12}p_{24}p_{31}+p_{14}p_{21}p_{32}-p_{14}p_{22}p_{31}, \\
&p_{13}p_{21}p_{34}+p_{13}p_{24}p_{31}+p_{14}p_{21}p_{33}-p_{14}p_{23}p_{31}, \\
&p_{12}p_{23}p_{34} - p_{12}p_{24}p_{33}-p_{13}p_{22}p_{34} +p_{13}p_{24}p_{32}+p_{14}p_{22}p_{33}-p_{14}p_{23}p_{32} \rangle.
\end{align*}
The MLdegree$(X_R)=13$.

When $u_{11}$ is a structural zero, we follow \cite{Rap06} and eliminate $p_{11}$ from the ideal $I(X)$ to obtain the defining ideal of $\pi_R(X)$, 
\begin{align*}
I(\pi_R(X)) = &\langle p_{12}p_{23}p_{34} - p_{12}p_{24}p_{33}-p_{13}p_{22}p_{34} +p_{13}p_{24}p_{32}+p_{14}p_{22}p_{33}-p_{14}p_{23}p_{32} \rangle.
\end{align*}
Optimizing over $\pi_R(X)$, yields 10 complex critical points, which matches the ML degree for $3\times 3$ rank $2$ matrices. 
 \end{example}

We now come to the description of the special fiber $pr_2^{-1}(u)$ when $u$ is a generic data vector in $U_S$, which connects this work with the likelihood correspondence  of \cite{HS13}.  

\begin{thm}\label{thm:subproblems}  Let $u$ be a generic data vector in $U_S$ for some $S \subseteq \{0, \ldots, n\}$. Let $X\subseteq\mathbb{P}^{n}$ be a codimension $c$ irreducible component of a projective variety defined by a reduced regular sequence of homogeneous polynomials $h_{1},\dots,h_{c}$. Let $X_R$ be a proper model zero variety for all $R \subseteq S$. Then,  the special fiber $pr_2^{-1}(u)$ contains the critical points of the problem $ML_{R,S}$ for all $R \subseteq S$.

Moreover, if $(p,u) \in pr_2^{-1}(u)$ with $p\in (X_R)_{reg} \setminus \cH_R$ then $p$ is a critical point of the problem $ML_{R,S}$ for some $R \subseteq S$.
\end{thm}

\begin{proof} 
Most of the work of this proof comes from the formulation of the Lagrange likelihood equations. First, note that for a variety $Y \subseteq \PP^n$ and $u \in U_{S'}$ for $S' \subseteq \{0,1,\ldots, n\}$, the point $p \in Y_{reg} \setminus \cH$ is a critical point on $Y$ for $l_{u,S'}$  if and only if the linear subspace $T_p^{\perp}$ contains the point $v \in \PP^{n-|R|}$ where 
$$v_i = \begin{cases}
    \frac{u_i}{p_i}-\frac{u_+}{p_+}   & \text{if } i \notin S, \\
      - \frac{u_+}{p_+} & \text{if } i \in S.
\end{cases}.$$
This condition results in the same equations as in $\LL(Y,u)$ when $u_i=0$ for all $i \in S$ and $p_i$ is assumed not to be zero when $i \notin S$.

Second, note that when we substitute $p_i=0$ in to $\LL(X,u)$, we get the equations for $\LL(X_R,u)$. Thus, by substituting $p_i = 0$ for $i \in R$ and $u_i=0$ for $i \in S$ into $\LL(X,u)$, we get a system of equations whose solutions are the critical points of $L_{u,S}$ on $X_R$.

Thus, if $X_R$ is a proper model zero variety and $p \in (X_R)_{reg}$,
then $p$ is a critical point on $(X_R)_{reg}$ for $L_{u,S}$ if and only if there exists a $\lambda$ such that $(p,\lambda)$ is an isolated solution to $\LL(X, u)$.

From Proposition \ref{prop:properties}, we know $u_i \neq 0$ implies $p_i \neq 0$, thus, we can account for all solutions to $\LL(X,u)$ since we consider every subset $R \subseteq S$.\end{proof}

In the proof of Theorem \ref{thm:subproblems}, we also proved the following statement (Proposition \ref{thm:dictionary}).  We state Proposition \ref{thm:dictionary} separately in order to highlight the equations for $ML_{R,S}$.

\begin{prop}\label{thm:dictionary}
Fix $u\in U_{S}$. Let $X\subseteq\mathbb{P}^{n}$ be a codimension $c$ irreducible component of a projective variety defined by a reduced regular sequence of homogeneous polynomials $h_{1},\dots,h_{c}$.
Whenever $X_R$ is proper, the critical points of $L_{u,S}$ restricted to $X_R$ are  regular isolated solutions
of the equations:
\begin{equation}\label{eq:LLE1R}
\begin{array}{c}
h_{1}=h_2=\cdots=h_{c}=0\\
p_i=0 \text{ for } i\in R, \text{ and }
\end{array}
\end{equation}
\begin{equation}\label{eq:LLE2R}
\begin{array}{rcll}
u_{+}&=&\left( \lambda_{1}\partial_{i}h_{1}+\lambda_{2}\partial_{i}h_{2}+\cdots +\lambda_c\partial_{i}h_{c}\right)& \text{ for } i\in S\setminus R\\
\left(u_{+}p_{i}-p_{+}u_{i}\right)&=&p_i \left( \lambda_{1}\partial_{i}h_{1}+\lambda_{2}\partial_{i}h_{2}+\cdots +\lambda_c\partial_{i}h_{c}\right) &\text{ for } i\not\in S\\
\end{array}
\end{equation}
Moreover, the solutions to \eqref{eq:LLE1R} and \eqref{eq:LLE2R} for all $R\subseteq S$ account for all the solutions to $\LL\left(X,u\right)$.
\end{prop}

An important consequence of  Theorem \ref{thm:subproblems} is  that we can use a parameter homotopy to take the solutions of $\LL(X,u)$ for $u\in U$ to the solutions of $\LL(X,v)$ for $v\in U_S$. Such methods are discussed in \cite{SW05}  and can be implemented in {\tt Bertini} \cite{BHSW06} or {\tt PHCpack} \cite{Ver99}. Doing so, we solve $2^{\mid S\mid}$ different optimization problems corresponding to the $2^{\mid S\mid }$ subsets of $S$. 
In the case $\mid S\mid=1$,  we get the following corollary. 

\begin{corollary} {\em \textbf{[ML degree bound]}} \label{prop:hsConjecture}
Suppose  $S=\{n\}$ and $X \subset \PP^n$ is an irreducible projective variety.  
Then for generic  $u\in U_S$, we have
\[ \mldegree(X) \geq \mldegree (X_S) +\mldegree (X, S) \]
Moreover, when $X$ is a generic complete intersection, the inequality becomes an equality.
\end{corollary}

\begin{proof} This follows from Theorem \ref{thm:subproblems} and the fact that the number of solutions to a parameterized family of polynomial systems for a generic choice of parameters can only decrease on nested parameter spaces (see Section 6.5 of \cite{BHSW13}).  Equality holds when $u$ remains off an exceptional subset $\mathcal{E} \subset U$ which is defined by an algebraic relation among the $p$ coordinates and $u$ coordinates \cite{SW05}[Theorem 7.1.1].  Since $X$ is a generic intersection, we have $U_S$ is not strictly contained in $\mathcal{E}$, and the equality holds.\end{proof}

As we can see from Corollary \ref{prop:hsConjecture}, solutions to $\LL(X, u)$ with $u \in U_S$ get partitioned into sampling zero and model zero solutions, in fact, we see this same behavior even as we increase the size of $S$.  We encode $\mldegree\left(X_{R},S\right)$ for all possible choices of $(R,S)$ in a table called the \emph{ML table} of $X$ whose rows are indexed by $R\subset \{0, 1, \ldots, n\}$ and whose columns are indexed by $S\subset \{0,1, \ldots, n\}$. Due to space considerations,  in our examples, we  often only print partial ML tables, i.e. that is subtables of the complete ML table.


\begin{example}\label{eg:rank2tables1}  Let $X \subset \PP^8$ be the projectivization of all $3 \times 3$ matrices of rank 2. A partial ML table of $X$ is below.
\[
\begin{array}{ccccc}
R\backslash S  & \{\} & \{11\} & \{12\} & \{11,12\}\\
\{\} & 10 & 5 & 5 & 1\\
\{11\} &  & 5 &  & 4\\
\{12\} &  &  & 5 & 4\\
\{11,12\} &  &  &  & 1
\end{array}.
\]
\end{example}

In Example \ref{eg:rank2tables1}, each  of the columns of the $\mltable (X)$ sum to $\mldegree (X)$. This does not happen for all varieties, but, in general, the column sums are lower bounds of the ML degree of $X$.

\begin{corollary}\label{cor:columns}
The column sums of the ML table of $X$ are less than or equal to 
$\mldegree(X)$, meaning  $$\mldegree(X)\geq\sum_{R\subseteq S}\mldegree(X_R,S).$$
Moreover, when $X$ is a generic complete intersection, the inequality becomes an equality.
\end{corollary}

The inequality in Corollary \ref{cor:columns} above can be strict as the next example shows. 
\begin{example}
Let $f=p_{0}^{3}+p_{1}^{3}+p_{2}^{3}+p_{3}^{3}$ define a hypersurface $X\subset\mathbb{P}^{3}$.
Some of the entries of the MLtable of $X$ are below. We have $\mldegree(X)=30$ but for $S=\{0,1\}$, we have $\sum_{R\subseteq S}\mldegree(X_R,S)=28$.

\[
\begin{array}{rcccc}
R\backslash  S & \{\} & \{0\} & \{0,1\} \\
\{\} & 30 & 21 & 12 \\
\{0\} &  & 9 & 7 \\
\{1\} &  &  & 7 \\
\{0,1\} &  &  & 2 \\
\end{array}
\]
\end{example}

\begin{rem}
Our definition for the entries of the ML table ignores multiplicities and singularities of the variety. We only take account regular isolated solutions.   An interesting research direction would be to take into account multiplicities to obtain an equality in the statement of Corollary \ref{prop:hsConjecture}. 
\end{rem}

We conclude this section with a full description of the ML table for a generic hypersurface of degree $d$ in $\mathbb{P}^n$.

\begin{thm}{\label{thm:formulaHypersurface}}
Suppose $X$ is a generic hypersurface of degree $d$ in $\mathbb{P}^{n}$
and let $s=\mid S\mid$ and $r=\mid R\mid$. Then 
\[
\mldegree\left(X_{R},S\right)=\begin{cases}
\frac{d}{d-1}\left(d^{n-s}-1\right) & s=r\\
d^{n- s+1}\left(d-1\right)^{s-r-1}, & s>r\\
0 & \text{otherwise. }
\end{cases}
\]
\end{thm}
\begin{proof}
Since the entries of the ML table of generic degree $d$ hypersurfaces  $X\subset\PP^n$ depend only on $d,n,$ and the
size of $R$ and $S$, we ease notation and let 
$$\shortM(r,s,n):=\mldegree\left(X_{R}\subset\mathbb{P}^{n},S\right).$$
By Proposition $\ref{thm:dictionary}$, it follows 
\begin{equation}\label{eq:goingDown}
\shortM(r,s,n+1)=\shortM(r-1,s-1,n)
\text{ for }r,s\geq1
\end{equation}
because a section of a generic hypersurface projected into a smaller projective space is again a generic degree $d$ hypersurface. 
We will use \eqref{eq:goingDown} to induct on $n$.

Recall by \cite{HKS05},  the ML degree of a generic degree $d$ hypersurface
in $\mathbb{P}^{n}$ is $\frac{d}{d-1}\left(d^{n}-1\right)$. 
When
$s=r$, we have $\shortM(r,s,n)=\frac{d}{d-1}\left(d^{n- s}-1\right)$
as desired. 
So for $n=2$,
$$\shortM(0,0,2)=\shortM(0,1,2)+\shortM(1,1,2).$$
Simple algebra reveals $\shortM(0,1,2)=d^{2}$.
With this we have shown  the theorem holds when $n=2$. 
To complete the proof by induction, we need only show 
$$\shortM(0,s,n+1)=d^{n- s+2}\left(d-1\right)^{s-r-1},\quad \text {for }0<s.$$ To show this we recall  
$$\mldegree( X\subset \PP^{n+1})=\sum_{R\subseteq S}\mldegree\left(X_{R}\subset\mathbb{P}^{n+1},S\right).$$
By induction and our work above, this equation becomes
\[
\begin{array}{ccl}
\shortM(0,0,n+1)&=&\shortM(0,s,n+1)+\shortM(s,s,n+1)+\\
&&\sum_{r=1}^{s-1}\binom{s}{r} \shortM(r-1,s-1,n)\\
\\
\frac{d(d^{n+1}-1)}{d-1}&=&\shortM(0,s,n+1)+\frac{d(d^{n+1-s}-1)}{d-1}+\\
&&\sum_{r=1}^{s-1}\binom{s}{r}d^{n-s+2}(d-1)^{s-r-1}.
\end{array}
\]
We solve for 
$\shortM(0,s,n+1)$ and use the binomial formula.\end{proof}


\section{Applications and Timings}

This section has five subsections.  The first subsection illustrates through examples how working with model zero varieties can decrease the required computing time for computing the ML degree.  The remaining four brief subsections focus on different applications: ML table homotopies, ML duality, tensors (multi-way tables), and Grassmannians.

\subsection{Timings}
In Examples \ref{ex:timings1} and \ref{ex:timings2} we compute a lower bound on the ML degree for several different varieties using the techniques from Section 2 and Section 3 and report on the timings.  All timings were done with a MacBook Pro having a 2.8 GhHz Intel Core i7 processor.

\begin{example}\label{ex:timings1}
Let $X$ be the hypersurface defined by
 $$p{}_{0}^{3}+2p_{1}^{3}+3p{}_{2}^{3}+5p{}_{3}^{3}+7p{}_{4}^{3}+11p{}_{5}^{3}+13p{}_{6}^{3}+17p{}_{7}^{3}.$$
 Using {\tt Macaulay2} and the variation of Algorithm 6 for complete intersections described in Section 4 of \cite{HKS05} to find the ML degree, we found that the computation did not complete within 24 hours.  The same was true using {\tt Macaulay2} to find the degree of the ideal defined by the Lagrange likelihood equations for a data vector with no zeros. But when the data vector had six zeros, we were able to make the following table of computations.

  
The columns of the table are labeled by the number of the six data zeros we considered as model zeros; the number of solutions to the system give by  \eqref{eq:LLE1R} and \eqref{eq:LLE2R}; and the timing of the computation using symbolic methods in \texttt{Macaulay2}. 
\[
\begin{array}{c}
\begin{array}{cccccccc}
\#R: & 0 & 1 & 2 & 3 & 4 & 5 & 6\\
\text{\# of Solutions:} & \textbf{3} & \textbf{9} & \textbf{18} & \textbf{36} & \textbf{72} & \textbf{144} & \textbf{288}\\
\text{Seconds:} & \leq 1 & \leq 1 & \leq 1 & \leq 1 & 4 & 85 & 2741
\end{array}
\end{array}
\]
Using this table and Corollary \ref{cor:columns}, the $\mldegree(X)$ is bounded below by
\[
1\cdot\textbf{3}+6\cdot\textbf{9}+15\cdot\textbf{18}+20\cdot\textbf{36}+15\cdot\textbf{4}+6\cdot\textbf{144}+1\cdot\textbf{288}=3279.
\]
\end{example}

\begin{example}\label{ex:timings2}
Let $X_n \subseteq \mathbb{P}^{3n-1}$ be defined by the $3\times 3$ minors of 
 the $3\times n$ matrix $[p_{ij}]$. In Section 4 of \cite{HRS12}, the ML degree of $X_n$ is conjectured to be $2^{n+1}-6$. The authors give supporting evidence up to $n=10$. 
 We add supporting evidence for $n=11,12,13, 14$.
 We take the data zeros to be $u_{13},u_{33}$, and $u_{2k}$ for $4\leq k\leq n$. 
Each column below contains $n$; the lower bound to the ML degree we compute; and computation time in minutes. 
Our computations were done symbolically in \texttt{Macaulay2} using the Lagrange Likelihood equations but with \eqref{eq:LLE1} taken to be all of the defining equations of $X$.
\[
\begin{array}{ccccccccc}
n:                        & 9       &10                   &11                  &12                   &13 & 14\\
\text{Bound:}   & 1018 & 2042 & \textbf{4090} & \textbf{8186}&\textbf{16378} & \textbf{32762} \\
\text{Min's:}     & \leq 1  & 4   & 12                & 30                 & 72                  & 194
\end{array}
\]
\end{example}

\subsection{ML table homotopy}

Let $X\subset\mathbb{P}^{n}$ be a generic complete intersection of
codimension $c$ defined by homogeneous polynomials $h_{1},\dots,h_{c}$. 
Let
$u$ be generic data vector in $U$, and let $u_{s}$
be a generic data vector in $U_{S}$ with $S \subseteq \{ 0, 1, \ldots, n\}$. 
Our first application of Corollary $\ref{prop:hsConjecture}$
is the construction of a homotopy to determine critical points of
$L_{u}$ on $X$. 
We determine the critical points of
$L_{u_{s},S}$ on $X\cap{\cal H}_{R}$ for each subset $R$ of $S$.
So rather than doing a single expensive computation to determine the
critical points of $L_{u}$ on $X$, we perform several
easier computations to determine critical points of $L_{u_{s},S}$. 
Doing so allows us to use Proposition \ref{thm:dictionary} to get the critical
points of $L_{u}$ using a coefficient-parameter homotopy.
The homotopy requires two steps.
Step 1 determines the start points by solving multiple systems of equations. 
Step 2 constructs the coefficient-parameter homotopy (see \cite{SW05}[\S 7]) that will do the path tracking.

\begin{example}
Let $X\subset\mathbb{P}^{3}$ be defined by $f=2p_{0}^{3}-3p_{1}^{3}+5p_{2}^{3}-7p_{3}^{3}$.
We note that $\mldegree (X)=39$ and the ML table of $X$ is:
\[
\begin{array}{rcccc}
R\backslash S & \{\} & \{0\} & \{1\} & \{0,1\}\\
\{\} & 39 & 27 & 27 & 18\\
\{0\} &  & 12 & -  & 9\\
\{1\} &  &  & 12 & 9\\
\{0,1\} &  &  &  & 3
\end{array}
\]
 Let
$S=\{0,1\}$ and let $u_{s}$ be a generic vector in $U_{S}$. For
Step 1 of the algorithm, we solve four systems of equations. Each system of equations corresponds to a choice of   $R$ from ${\cal R}:=\left\{ \emptyset,\left\{ 0\right\} ,\left\{ 1\right\} ,\left\{ 0,1\right\} \right\} $.
For example, when $R=\left\{ 0,1\right\} $, we solve the following system
\[
\begin{array}{c}
f=0,\,p_{0}=0\,p_{1}=0\\
\left(u_{+}p_{2}-p_{+}u_{2}\right)=p_{2}\lambda_{1}\cdot \partial_{2}f\\
\left(u_{+}p_{3}-p_{+}u_{3}\right)=p_{3}\lambda_{1}\cdot \partial_{3}f
\end{array}
\]
and find   $3$ solutions. In general, we solve the equations
in Proposition \ref{thm:dictionary}. So when $R=\emptyset,\{0\},\{1\},\{0,1\}$ we determine there
are $18,9,9,3$ solutions for the respective systems for a total
of $39$ solutions. For Step 2, by Proposition $\ref{thm:dictionary}$, the computed $39$
solutions are solutions to the Lagrange likelihood equations $\LL\left(X,u_{s}\right)$.
So by using the coefficient-parameter homotopy $\LL(X,u_s\to u)$, we can go from data with zeros $u_{s}$
to generic data $u$. 
\end{example}

\begin{alg}\label{alg:1}$\ $
\begin{itemize}
\item Input $u_s\in U_S$ and homogeneous polynomials $h_1,h_2,\dots ,h_c$ defining  $X$ with codimension $c$.
\item (Step 1) Solve $\LL(X_R,u_s)$ for each $R\subset S$ to determine the start points of the homotopy. 
\item (Step 2) Construct and solve the coefficient-parameter homotopy $\LL(X,u_s\to u)$. 
\item Output  solutions to $\LL(X,u)$ yielding the critical points of $L_u$ on $X$.
\end{itemize}
\end{alg}

The possible advantage of this homotopy is that we may be able to get several
critical points of $L_{u}$ quicker than with the standard algebraic method.  
Thus, when other methods fail, we can still get some insight if the ML degree of $X$ is small. 
Moreover, one can 
use monodromy methods \cite{SVW01} to attempt to recover additional solutions.
 One drawback is that by increasing the size of $S$
we also increase the number of subproblems we need to solve,  a
second drawback is that we may not know \emph{a priori} that $\sum_{R\subseteq S}\mldegree(X_R,S)$
equals the ML degree. To address the first drawback, one can take advantage
of the structure of the problem to lessen the number of subproblems.
For example, in the case when $X$ is a generic hypersurface, we know that the ML degree
of $X$ depends only on the size of $R$ and $S$. Taking advantage
of this structure and pairing change of variables with parameter homotopies, we preprocess much fewer subproblems---namely $|S|$ subproblems
versus $2^{\mid S\mid}$. While we do not have equality in Corollary \ref{prop:hsConjecture} in general, equality does occur in some examples (see Theorem \ref{thm:formulaHypersurface}).

\subsection{Maximum likelihood duality}
In this section, we extend ML duality for matrix models when $u$ contains zero entries. 
We let $X\subset\PP^{mn-1}$ be the variety of $m\times n$ matrices $[p_{ij}]$ of rank less than or equal to $r$ and we let  $Y\subset\PP^{mn-1}$ be the variety of $m\times n$ matrices $[q_{ij}]$of rank less than or equal to $m-r+1$ where $m\leq n$.
In \cite{DR12}, it is shown that $\mldegree X=\mldegree Y$ by considering critical points of  $l_u$ on subvarieties of the algebraic torus. A bijection between said critical points is also given. 
Translating  these results  into the language of determining  critical points of $L_u$ on subvarieties of projective space, we are able to talk about sampling zeros and model zeros. 

Before the proposition  we introduce the following notation. We have  $p_{i+}:=p_{i1}+\cdots +p_{in}$ and 
$p_{+j}:=p_{1j}+\cdots +p_{mj}$. We  also define     $u_{i+}$ and $u_{+j}$ analogously.
In addition, we take $*$ to be the Hadamard (coordinate-wise) product between two matrices. 
For example, $[p_{ij}]*[q_{ij}]=[p_{ij}q_{ij}]$.
\begin{prop} \label{thm:duality} Let $X$ and $Y$ be defined as above so that they are ML dual varieties. 
Let $S \subset [n]$ and $u \in U_S$. If $P \in \C^{mn}$ is a solution to $\LL(X, u)$, then there exists a $Q \in \C^{mn}$ such that $Q$ is a solution to $\LL(Y,u)$ and
\begin{equation}\label{eq:hadamard}
P \star Q = \Omega_U 
\end{equation}
\begin{equation} \label{eq:omega}
\text{ where }  \Omega_U= \left[\frac{u}{u_{++}}\right] \star\left[ \frac{u_{i+}u_{+j}}{u_{++}^2}\right] 
\end{equation}
\end{prop}

\begin{proof} Let $\mathcal{D} \subset \PP^{nm-1} \times \PP^{nm-1} \times \PP^{nm-1}$ be the set of all points $(p, q, u)$ such that $(p,u) \in \cL_X$, $(q,u) \in \cL_Y$ and
 \[ u_{++}^3p_{ij}q_{ij}-p_{++}q_{++}u_{i+}u_{ij}u_{+j}=0 \text{ for } 0 \leq i \leq m, \ 0 \leq i \leq n. \]
 The set $\mathcal{D}$ is a projective variety, thus, if we consider the projection  
 \begin{align*}
 \phi : \PP^{n} \times \PP^{n} \times \PP^{n} & \to \PP^{n} \times \PP^n\\
(p,q,u) &\mapsto (p,u),
\end{align*}
the image of $\mathcal{D}$ under $\phi$ is a variety.  By Theorem 1 of \cite{DR12}, we know that  a dense open subset of $\cL_{X}$ is contained in $\phi(\mathcal{D})$, therefore, $\cL_{X} \subseteq \phi(\mathcal{D})$ and the statement of the theorem follows.\end{proof}

\begin{thm}\label{cor:dual}
Let $X$ and $Y$ be defined as in Lemma \ref{thm:duality}. Fix $S\subset [m]\times[n]$ and generic $u\in U_S$.
Then a solution to the maximum likelihood estimation problem  $ML_{R,S}(u)$ is dual to  a solution to the  maximum likelihood estimation problem   $ML_{R',S}(u)$, with $(S\setminus R )\subset R'$.
\end{thm}

When $|S|=1$, the theorem says that a sampling zero critical point is dual to a model zero critical point.
We also believe that the converse, model zero critical points are dual to sampling zero critical points is true, and that in general,  $(S\setminus R )\subset R'$ is actually an equality in the theorem. 
Nonetheless, because computing model zeros is heuristically easier than computing sampling zeros, we believe computational gains can be made with Theorem \ref{cor:dual}.

In Example \ref{eg:rank2tables1}, we see that a column of the ML table is symmetric. This is because the variety of $3\times 3$ matrices of rank $2$ is ML self dual. Other examples of varieties that are ML self dual include $m \times n$ matrices of rank $\frac{m+1}{2}$ with $m$ being odd. 
We conclude this subsection with an partial ML table of $4\times 4$ matrices of rank $2$ and of rank $3$.
$$\begin{array}{c}
\text{ML table of } 4 \times 4 \text{ rank } 2 \text{ matrices} \\
\begin{array}{rcccc}
R\backslash\backslash S & \left\{ \right\}  & \left\{ 11\right\}  & \left\{ 11,44\right\}  & \left\{ 11,22,44\right\}   \\
\left\{ \right\}  & 191 & 118 & 76 & 51 \\
\left\{ 11\right\}  &  & 73 & 42 & 25 \\
\left\{ 22\right\}  &  &  &  & 25 \\
\left\{ 44\right\}  &  &  & 42 & 25 \\
\left\{ 11,22\right\}  &  &  &  & 17 \\
\left\{ 11,44\right\}  &  &  & 31 & 17 \\
\left\{ 22,44\right\}  &  &  &  & 17 \\
\left\{ 11,22,44\right\}  &  &  &  & 14 \\
\end{array}
\end{array}
$$ $$
\begin{array}{c}
\text{ML table of } 4 \times 4 \text{ rank } 3 \text{ matrices} \\
\begin{array}{rcccc}
R\backslash\backslash S & \left\{ \right\}  & \left\{ 11\right\}  & \left\{ 11,44\right\}  & \left\{ 11,22,44\right\}   \\
\left\{ \right\}  & 191 & 73 & 31 & 14 \\
\left\{ 11\right\}  &  & 118 & 42 & 17 \\
\left\{ 22\right\}  &  &  &  & 17 \\
\left\{ 44\right\}  &  &  & 42 & 17 \\
\left\{ 11,22\right\}  &  &  &  & 25 \\
\left\{ 11,44\right\}  &  &  & 76 & 25 \\
\left\{ 22,44\right\}  &  &  &  & 25 \\
\left\{ 11,22,44\right\}  &  &  &  & 51 \\
\end{array}
\end{array}
$$  
An ongoing project is to give recursive formulas for entries of the  ML table of $m \times n$ matrices of rank $r$.


\subsection{Tensors}\label{tensors}
Let $T$ be the variety of $2 \times 2 \times 2 \times 2$ tensors of the form $[p_{ijkl}$ ]with border rank $\leq 2$. The ML degree of this variety is unknown. The variety is defined by the $3 \times 3$ minors of all possible flattenings. This is an overdetermined system of equations with codimension $6$.  We choose $6$ of the equations to be $h_1,\dots,h_6$ for the  Lagrange likelihood equations. 
For the model zero variety with $p_{1111}=p_{2222}=0$ we find $3$ solutions for a generic $u\in U_S$ with $S=\{1111,2222\}$.
When we solve the Lagrange likelihood equations for $R=\{1111\}$, we find $52$ solutions with $p\in X$. 

\begin{thm} If $T$ is as above, then   the ML degree of T is greater than or equal to $52$.
\end{thm}

In this example, we also see that when we have data with zeros the number of critical points can drop significantly as we introduce more model zeros. 

\subsection{Grassmannians}

Let the ideal $I_{2,n}$ be generated by the quadrics 
\[
p_{ij}p_{kl}-p_{ik}p_{jl}+p_{il}p_{jk},\quad1\leq i<j<k<l\leq n.
\]
Then the variety of $I_{2,n}$ is the Grassmannian $\Gr_{2,n}\subset\mathbb{P}^{\binom{n}{2}-1}$.
The Grassmannian $\Gr_{2,n}$ parameterizes lines in the projective
space $\mathbb{P}^{n-1}$. 
Below we have a table of computations. The top line consists  of ML degrees of $\Gr_{2,n}$. The next two lines are entries of the ML table  for one zero in the data.
\[
\begin{array}{rcccc}
 & \Gr_{2,4} & \Gr_{2,5} & \Gr_{2,6} \\
\mldegree X & 4 & 22 & 156 \\
\mldegree\left(X_{\{12\}},\{12\}\right) & 1 & 4 & 22 \\
\mldegree\left(X_{\emptyset},\{12\}\right) & 3 & 18 & 134 
\end{array}
\]
These computations were performed by choosing $c=\codim X$ generators of $I_{2,n}$ to be $h_1\dots h_c$  for $\LL(X,u)$. We used the numerical software $\tt bertini$ and symbolic packages available in $\tt M2$ \cite{GS}. 
From this data we make the following conjecture to motivate the pursuit of a  recursive formula for ML degrees of Grassmannians.
\begin{conjecture}
For $n\geq 4$ we conjecture  
$$\mldegree \Gr_{2,n}=\mldegree (\Gr_{2,n+1}\cap \{p_{12}=0\}).$$
\end{conjecture}

\section{Conclusion}
Understanding model and sampling zeros gives us insights into the maximum likelihood degree for a given model.  When the data vector contains a zero entry, we see that critical points to the likelihood function partition into two groups: critical points for the sampling zero problem and critical points for the model zero problem. This split can help us obtain bounds for the ML degree and provides interesting directions for further research within the study of likelihood geometry, for example, determining which varieties yield an equality in Corollary \ref{prop:hsConjecture}.  Furthermore, model zeros can help with the computational problem of finding all the solutions to a set of likelihood equations. This paper illustrates some of the advantages of working with model zeros, as seen by the lower bound obtained on the set of $2 \times 2 \times 2 \times 2$ tensors of border rank $\leq 2$.  We hope that the problem mentioned in Section 4.2 of determining whether model zero critical points are dual to sampling zero critical points is furthered explored.  A positive answer would yield significant gains in understanding the ML degree for determinantal varieties.


\section{ Acknowledgements}
The authors would like to thank Kaie Kubjas, Anton Leykin, Bernd Sturmfels,  and Seth Sullivant for their helpful  suggestions and comments.



\end{document}